\documentclass[11pt]{article}

\usepackage[T1]{fontenc}
\usepackage[utf8]{inputenc}
\usepackage{lmodern}
\usepackage{amsmath,amsthm,amssymb,amsfonts,mathtools}
\usepackage{hyperref}
\usepackage{geometry}
\newtheorem{theorem}{Theorem}[section]
\newtheorem{lemma}[theorem]{Lemma}
\newtheorem{proposition}[theorem]{Proposition}
\newtheorem{corollary}[theorem]{Corollary}

\theoremstyle{definition}
\newtheorem{definition}[theorem]{Definition}
\newtheorem{remark}[theorem]{Remark}

\newcommand{\Av}{\mathrm{Av}}
\newcommand{\X}{\mathsf{X}}

\newcommand{\Fv}{\mathcal{F}_v}

\newcommand{\beq}{\begin{eqnarray*}}
	\newcommand{\feq}{\end{eqnarray*}}
\newcommand{\beqn}{\begin{eqnarray}}
	\newcommand{\feqn}{\end{eqnarray}}

\newcommand{\keywords}[1]{%
	\par\noindent\textbf{Keywords. }#1\par
}
\newcommand{\subjclass}[1]{\noindent\textbf{MSC 2020. }#1}

\title{Growing Avoiders from the Right: An Operator-Theoretic Approach}
\author{Reza~Rastegar\thanks{Woodinville, WA; e-mail: reza.j.rastegar@gmail.com}}

\begin{document}
	\maketitle
	
	\begin{abstract}
		Marcus and Tardos \cite{MarcusTardos2004} proved the Stanley--Wilf conjecture by
		reducing pattern avoidance to an extremal problem on $0$--$1$ matrices. We recast
		the same exponential-growth conclusion for classical permutation patterns inside
		a ``grow from the right'' operator-theoretic framework. A $v$-avoiding
		permutation is built by right insertion; at each step we keep a pruned family of
		locations of $(k{-}1)$-partial occurrences of $v$ (the \emph{frontier}), each
		carrying its forbidden rank interval. The insertion step then induces a
		nonnegative transfer operator on a doubly weighted $\ell^\infty$ space. A
		quadratic penalty in the length makes this operator bounded, and a Neumann-series
		argument on a natural separable predual yields a convergent operator-valued
		Neumann series and an operator-theoretic representation of the growth series in
		a disc; combined with the Marcus--Tardos exponential bound, this recovers the
		Stanley--Wilf finiteness conclusion for $\Av(v)$.
		
		The formulation is completely internal; we never pass to $0$--$1$ matrices—and it
		cleanly separates the pattern-dependent combinatorics of the frontier from a
		purely operator-theoretic core. In particular, we obtain an abstract
		``right-insertion/transfer-operator'' theorem: for any system whose frontier grows at
		most linearly and whose transfer operator satisfies a uniform quadratic
		length bound, the associated transfer operator is quasi-compact on a doubly
		weighted $\ell^\infty$ space, and the Neumann series for its adjoint converges
		in a disc. In particular, all growth series obtained by pairing with continuous
		functionals on the dual are analytic (Theorem~\ref{thm:RV-abstract}).
	\end{abstract}
	
	
	
	
	\section{Introduction}
	
	Permutation patterns form one of the central testing grounds for modern enumerative and structural combinatorics; see, for instance, the foundational work of Simion--Schmidt, B\'ona, Kitaev, and Vatter \cite{SimionSchmidt1985,BonaBook,KitaevBook,VatterSurvey}. 
	The starting point of this paper is the now-classical fact that for every fixed permutation pattern $v$ the number
	\[
	|\Av_n(v)| = |\{ \pi \in S_n : \pi \text{ avoids } v \}|
	\]
	grows at most exponentially fast in $n$. This is the content of the Stanley--Wilf conjecture, posed independently by Stanley and Wilf in the early 1980s and settled by Marcus and Tardos \cite{MarcusTardos2004} -- see also Theorem~\ref{thm:MT-operator}. Their proof recasts pattern avoidance as an extremal problem on $0$--$1$ matrices, proves a strengthened form of the F\"uredi--Hajnal conjecture, and then reads off exponential growth. Because the argument is short, elementary, and sharp enough to imply all then-known special cases (e.g.\ \cite{SimionSchmidt1985,Arratia1999,BonaBook}), it has become the standard reference; see Vatter's survey \cite{VatterSurvey} for a modern account and for the role of Marcus--Tardos in the structural theory of permutation classes.
	
	Our aim is not to improve on the Marcus--Tardos bound, nor to bypass matrix methods for their own sake, but to show that the same exponential-growth phenomenon \emph{for classical patterns} can be obtained inside a framework that is already familiar in enumerative combinatorics and in dynamical systems: ``grow the object from a boundary, keep a bounded amount of local data, and let a transfer operator act on a weighted function space.'' Variants of this philosophy occur in the transfer-matrix method for walks on strips \cite[Ch.~4]{FlajoletSedgewick}, in column-by-column growth of polyominoes and related lattice objects (see e.g.\ \cite{BousquetMelou1996,BousquetMelouFeretic1997}), and in the use of Ruelle--Perron--Frobenius operators to encode symbolic extensions of dynamical systems \cite{BaladiBook,Keller1998}. In all those settings one ends up with a linear operator $T$ on a Banach space of functions on the state space, and the exponential growth rate of the combinatorial objects is the spectral radius of~$T$.
	
	What we show here is that classical permutation pattern avoidance admits exactly such a formulation. We build a permutation by repeatedly appending a rightmost position and choosing its value by rank (the usual right-to-left reconstruction); to ensure that the new letter does not complete an occurrence of $v$, we keep at each step a pruned family of \emph{locations of $(k{-}1)$-partial occurrences} of $v$, which we call the \emph{frontier}. Each such partial occurrence determines a contiguous interval of ranks for the new letter that would extend it to a full $k$-pattern at the next step, so the frontier implicitly encodes the corresponding forbidden-rank intervals. (Shorter partial occurrences of $v$ play no direct role in the next step: they must first grow to length $k{-}1$ before they can become dangerous, so tracking them would only inflate the state without adding information.) Two mild obstacles appear:
	
	\begin{itemize}
		\item for every active $(k{-}1)$-partial occurrence of $v$, the frontier must remember \emph{where} it sits in the current permutation and which ranks would extend it, so the state is not just the permutation but the permutation \emph{plus} this pruned family of partial occurrences (or, equivalently, the associated forbidden-rank intervals);
		\item the number of legal ranks grows linearly with the length, so a single weight in the length is not enough to make the transfer operator bounded.
	\end{itemize}
	
	We deal with the first point by making the frontier explicit at the level of $(k{-}1)$-partial occurrences, computing from each of them the interval of ranks that would complete it, and discarding those whose completion interval is empty (they can never become dangerous at the very next step). We deal with the second by placing our operator on a doubly weighted $\ell^\infty$ space in which we penalize both the frontier size and \emph{quadratically} the length. With these choices the insertion operator becomes bounded on a doubly weighted $\ell^\infty$ space. This suffices, via a dual Neumann-series argument, to obtain analyticity of the counting series in a disc and hence the Stanley--Wilf finiteness conclusion for classical patterns (Theorem~\ref{thm:MT-operator}). As an additional structural result (which we record but do not use for the finiteness statement itself), a standard core/tail decomposition of Ionescu Tulcea--Marinescu \cite{IonescuTulceaMarinescu1950}, in the streamlined form of Hennion \cite{Hennion1993}, shows that the operator is in fact a finite-rank perturbation of a contraction and therefore quasi-compact.
	
	This reformulation has two concrete advantages.
	
	First, it separates the \emph{combinatorial} task from the \emph{analytic} one. On the combinatorial side we must specify how, under right insertion, we keep track of the ``dangerous'' $(k{-}1)$-partial occurrences of the pattern. For classical patterns this comes down to defining the frontier and proving that it behaves monotonically under legal insertions (Lemma~\ref{lem:frontier-monotone}), in the sense that existing partial occurrences (and their associated forbidden-rank intervals) are transported but not destroyed. Once this is in place, the rest of the argument---choice of weights, duality, and the use of a Neumann series on the Banach dual---is pattern-independent.
	
	Second, the operator viewpoint packages the Stanley--Wilf finiteness statement into a reusable template: grow the permutation by right insertion, choose a complexity measure on decorated states, and apply general transfer-operator machinery to read off exponential growth from spectral data. This makes transparent which ingredients are genuinely pattern-specific (the organization of the frontier) and which are universal (the weighted $\ell^\infty$ framework and the dual Neumann-series argument). It also suggests how one might try to extend the method to certain ``right-visible'' generalizations of classical patterns, where all constraints triggered by adding a new rightmost entry are already visible at the moment of insertion; we return briefly to this perspective in the concluding remarks.
	
	\medskip
	
	Beyond the mere existence of a finite exponential growth rate, one would like to understand the \emph{precise} asymptotic behaviour of $|\Av_n(v)|$, for instance whether it always admits an expansion of the form
	\[
	|\Av_n(v)| \;\sim\; c_v\,n^{\alpha_v}\,\rho_v^n
	\]
	for some constants $c_v>0$, $\alpha_v\in\mathbb{R}$ and $\rho_v>0$. Our framework does not address this question. In special cases, such as the monotone pattern $12\cdots k$, sharp asymptotics of this shape are known by completely different methods (via the representation theory of the symmetric group and the RSK correspondence; see Regev~\cite{Regev1981}). From the transfer-operator viewpoint, such refinements would require a much more delicate analysis of the peripheral spectrum and the singularities of the resolvent for $T_{v,1}$ or its adjoint, and we do not pursue this direction here.
	
	The paper is organized as follows. Section~\ref{sec:perm-specialization} sets up the right-insertion picture for classical pattern avoidance, defines the frontier as a pruned family of $(k{-}1)$-partial occurrences together with their forbidden-rank intervals, and makes its monotonicity under legal insertions fully explicit, including an explicit ``frontier blocks what it should block'' lemma. Section~\ref{sec:penalty-operator} introduces the two-parameter penalty and proves that the insertion operator is bounded even at $z=1$ after shrinking the penalty. Section~\ref{sec:quasicompact} develops a separable predual and a dual transfer-operator point of view from which we read off an analytic counting series; along the way we also record the standard core/tail decomposition and quasi-compactness of the operator, although these structural results are not needed for Theorem~\ref{thm:MT-operator}. We conclude in Section~\ref{sec:concluding} with remarks on how far the ``right-visible'' condition might be pushed beyond classical patterns, and how this operator-theoretic viewpoint sits with the broader theory of permutation classes \cite{VatterSurvey}.

	\section{Permutation pattern avoidance}\label{sec:perm-specialization}
	
	In this section we formalize the right-insertion viewpoint for classical permutations avoiding a fixed pattern. The idea is to view a growing permutation as a state, and the admissible ways of inserting a new value as legal extensions, so that our transfer-operator machinery can be applied directly to the pattern-avoidance problem.
	
	More concretely, we will build permutations one value at a time by inserting a new entry at the right end, with an insertion rule that is in bijection with the Lehmer-code encoding of permutations. Pattern-avoidance will then be enforced by restricting to those insertion choices that keep the forbidden pattern out; the complexity function will be chosen so that it captures the amount of ``forbidden-pattern structure'' already present in the growing permutation. This viewpoint allows us to rewrite the avoidance problem as the study of a suitable transfer operator on a state space of decorated permutations.
	
	\subsection{Right insertion, its formal definition, and the Lehmer code}
	
	Among the many possible ways to grow a permutation (left insertion, insertion into arbitrary positions, inflation operations, \emph{etc.}), right insertion with rank parameters turns out to be particularly convenient: it is compatible with the factorial number system, it produces a simple product structure on the space of insertion codes, and it fits directly into our operator framework via the map $\Phi$ and the local extension sets $A(x)$. We therefore fix a canonical right-insertion scheme and recall its connection to the classical Lehmer code.
	
	Fix a classical pattern $v\in S_k$ with $k\ge 2$. We build a permutation $\pi$ by \emph{right insertion}:
	
	\begin{definition}[Right insertion]\label{def:right-insertion}
		Let $\pi\in S_n$ and let $r\in \{1,\dots,n+1\}$. Define $\mathrm{ins}(\pi,r)\in S_{n+1}$ to be the permutation obtained by
		\begin{enumerate}
			\item appending a new rightmost position at index $n+1$;
			\item giving it value $r$;
			\item and increasing by $1$ every existing value of $\pi$ that is $\ge r$.
		\end{enumerate}
	\end{definition}
	
	It is a standard fact that this right-insertion encoding is \emph{bijective}. We omit the proof and refer the reader interested in learning more about encodings of permutations to, for instance, \cite{BonaBook}.
	
	\begin{lemma}[Right insertion and Lehmer code]\label{lem:lehmer}
		For every $n\ge 0$, the map
		\[
		(r_1,\dots,r_n)\in \prod_{j=1}^n \{1,\dots,j\}
		\longmapsto
		\mathrm{ins}(\cdots \mathrm{ins}(\mathrm{ins}(\varnothing,r_1),r_2)\cdots,r_n)
		\]
		is a bijection from $\prod_{j=1}^n \{1,\dots,j\}$ onto $S_n$. In particular, there are $n!$ such insertion sequences of length $n$, and each $\pi\in S_n$ is produced by \emph{exactly one} of them.
	\end{lemma}

	\subsection{Frontiers for classical patterns}\label{subsec:frontier-classical}
	
	To enforce $v$-avoidance during right insertion, we must block exactly those
	ranks whose insertion would \emph{complete} an occurrence of $v$. A key
	simplification is that only \emph{partial occurrences of length $k-1$} can be
	completed to a full $k$-pattern. Shorter prefixes play no role at the next
	step: they may later grow into $(k-1)$-prefixes after further insertions, but
	they cannot become dangerous immediately, and tracking them would artificially
	inflate the frontier without providing any additional safety.
	
	We therefore adopt the following compressed, location-based definition of the frontier.
	
	\begin{definition}[Frontier for a classical pattern]\label{def:frontier-v}
		Let $v\in S_k$ with $k\ge 2$, and let $\pi\in S_n$ avoid~$v$.
		
		A \emph{$(k{-}1)$-partial occurrence} of $v$ in $\pi$ is a choice of indices
		\[
		i_1<\dots<i_{k-1}
		\]
		such that $(\pi_{i_1},\dots,\pi_{i_{k-1}})$ is order-isomorphic
		to the prefix $(v_1,\dots,v_{k-1})$.
		
		For such a $(k{-}1)$-partial occurrence, define its \emph{forbidden rank set}
		\[
		J_\pi(i_1,\dots,i_{k-1}) \subseteq \{1,\dots,n+1\}
		\]
		by
		\[
		J_\pi(i_1,\dots,i_{k-1}) \;:=\; \bigl\{ r \in \{1,\dots,n+1\} :
		\pi' = \mathrm{ins}(\pi,r) \text{ satisfies }
		(\pi'_{i_1},\dots,\pi'_{i_{k-1}},\pi'_{n+1})
		\sim (v_1,\dots,v_k) \bigr\},
		\]
		i.e.\ $J_\pi(i_1,\dots,i_{k-1})$ is the set of ranks $r$ for which inserting $r$ uses the new rightmost position to complete a full copy of $v$ at the next step.
		As we will see below (Lemma~\ref{lem:frontier-intervals}), each such set $J_\pi(i_1,\dots,i_{k-1})$
		is in fact an interval of consecutive ranks. We allow this set to be empty.
		
		The \emph{frontier} of $\pi$ is the set
		\[
		\Fv(\pi)
		\;:=\;
		\bigl\{(i_1,\dots,i_{k-1}) :
		i_1<\dots<i_{k-1} \text{ is a $(k{-}1)$-partial occurrence of $v$ in $\pi$ and }
		J_\pi(i_1,\dots,i_{k-1})\neq\varnothing
		\bigr\}.
		\]
		Thus $\Fv(\pi)$ is a finite family of locations of $(k{-}1)$-partial occurrences; each element of the frontier carries, implicitly, its forbidden rank-interval $J_\pi(i_1,\dots,i_{k-1})$. The union
		\[
		\mathrm{Forb}_v(\pi) \;:=\; \bigcup_{(i_1,\dots,i_{k-1})\in\Fv(\pi)} J_\pi(i_1,\dots,i_{k-1})
		\]
		is the set of \emph{forbidden ranks}, and we write
		\[
		A_v(\pi) \;:=\; \{1,\dots,n+1\}\setminus \mathrm{Forb}_v(\pi)
		\]
		for the set of \emph{frontier-legal ranks}.
	\end{definition}
	
	The following lemma records finiteness and the basic soundness of the
	construction.
	
	\begin{lemma}\label{lem:frontier-finite-sound}
		Let $v\in S_k$ and $\pi\in S_n$ avoid~$v$. Then:
		\begin{enumerate}
			\item $\Fv(\pi)$ is finite;
			\item if $r\notin\mathrm{Forb}_v(\pi)$, then $\mathrm{ins}(\pi,r)\in\Av(v)$.
		\end{enumerate}
	\end{lemma}
	
	\begin{proof}
		(1) There are at most $\binom{n}{k-1}$ choices of $(k{-}1)$-tuples
		$i_1<\dots<i_{k-1}$, hence only finitely many $(k{-}1)$-partial occurrences and therefore only finitely many frontier elements.
		
		(2) By construction, if $r\in\mathrm{Forb}_v(\pi)$ then there is a
		$(k{-}1)$-partial occurrence $(i_1,\dots,i_{k-1})$ whose forbidden rank set $J_\pi(i_1,\dots,i_{k-1})$ contains $r$, and
		inserting $r$ completes a copy of $v$ using the new rightmost position.
		Thus $\mathrm{ins}(\pi,r)\notin\Av(v)$ whenever
		$r\in\mathrm{Forb}_v(\pi)$.
		
		It remains to prove the converse: if $\mathrm{ins}(\pi,r)$ contains a copy
		of $v$, then necessarily $r\in\mathrm{Forb}_v(\pi)$.
		
		Let $\pi'=\mathrm{ins}(\pi,r)$ and suppose $\pi'$ contains $v$. Then there
		exist indices
		\[
		j_1<\dots<j_k
		\]
		such that $(\pi'_{j_1},\dots,\pi'_{j_k})\sim v$.
		If $j_k\neq n+1$, then all $j_\ell\le n$, so the same $k$-tuple of
		positions already existed in $\pi$. Since right insertion preserves the
		relative order among the old entries, we would have
		$(\pi_{j_1},\dots,\pi_{j_k})\sim v$, contradicting $\pi\in\Av(v)$.
		
		Hence $j_k=n+1$ and $j_1<\dots<j_{k-1}\le n$. Again, because the insertion
		$\mathrm{ins}(\pi,r)$ only increases some of the old values but does not
		change their relative order, the subsequence
		$(\pi_{j_1},\dots,\pi_{j_{k-1}})$ is order-isomorphic to the prefix
		$(v_1,\dots,v_{k-1})$. Thus $(j_1,\dots,j_{k-1})$ is a $(k{-}1)$-partial
		occurrence of $v$ in $\pi$.
		
		By Definition~\ref{def:frontier-v}, the forbidden rank set
		\[
		J_\pi(j_1,\dots,j_{k-1})
		\]
		consists exactly of those $r$ for which inserting at rank $r$ completes a copy of $v$ using the new rightmost position. Since
		our chosen $r$ does precisely that, we have
		\[
		r\in J_\pi(j_1,\dots,j_{k-1}) \subseteq\mathrm{Forb}_v(\pi).
		\]
		
		Therefore, if $r\notin\mathrm{Forb}_v(\pi)$, then $\mathrm{ins}(\pi,r)$
		cannot contain an occurrence of $v$, i.e.\ $\mathrm{ins}(\pi,r)\in\Av(v)$.
	\end{proof}
	
	We now justify the claim that each forbidden rank set is an interval of
	consecutive integers.
	
	\begin{lemma}\label{lem:frontier-intervals}
		Let $v\in S_k$ and $\pi\in\Av(v)$, and fix a $(k{-}1)$-partial occurrence
		$i_1<\dots<i_{k-1}$ of $v$ in $\pi$. Let
		\[
		J := J_\pi(i_1,\dots,i_{k-1})
		\]
		be its forbidden rank set as in Definition~\ref{def:frontier-v}. Then $J$ is an (integer) interval: there
		exist integers $1\le a\le b\le n+1$ such that
		\[
		J \;=\; \{a,a+1,\dots,b\}.
		\]
	\end{lemma}
	
	\begin{proof}
		Write $x_\ell := \pi_{i_\ell}$ for $\ell=1,\dots,k-1$. For each
		$r\in\{1,\dots,n+1\}$, the right insertion $\pi'=\mathrm{ins}(\pi,r)$
		leaves the relative order among the old entries $(x_1,\dots,x_{k-1})$
		unchanged, and replaces the new point by the value $\pi'_{n+1}=r$.
		Thus the pattern of
		\[
		(\pi'_{i_1},\dots,\pi'_{i_{k-1}},\pi'_{n+1})
		\;=\;
		(\pi'_{i_1},\dots,\pi'_{i_{k-1}},r)
		\]
		is determined entirely by the relative order of $r$ with each $x_\ell$.
		
		Fix $\ell\in\{1,\dots,k-1\}$, and consider the comparison between the new
		point and $x_\ell$ in $\pi'$. After insertion we have
		\[
		\pi'_{i_\ell} \;=\;
		\begin{cases}
			x_\ell+1, & \text{if } x_\ell \ge r,\\
			x_\ell,   & \text{if } x_\ell < r.
		\end{cases}
		\]
		A direct check shows:
		\begin{itemize}
			\item $r < \pi'_{i_\ell}$ (new point below $x_\ell$) if and only if
			$r \le x_\ell$;
			\item $r > \pi'_{i_\ell}$ (new point above $x_\ell$) if and only if
			$r \ge x_\ell+1$.
		\end{itemize}
		Thus, for each $\ell$, the relation ``new below $x_\ell$'' or
		``new above $x_\ell$'' imposes a \emph{single-sided} constraint on $r$ of
		the form $r\le x_\ell$ or $r\ge x_\ell+1$.
		
		Now compare with the pattern $v$. For each $\ell$ we require that the order
		relation between $v_k$ and $v_\ell$ matches that between $r$ and
		$\pi'_{i_\ell}$:
		\begin{itemize}
			\item If $v_k < v_\ell$, then we need $r < \pi'_{i_\ell}$, i.e.\
			the constraint $r \le x_\ell$.
			\item If $v_k > v_\ell$, then we need $r > \pi'_{i_\ell}$, i.e.\
			the constraint $r \ge x_\ell+1$.
		\end{itemize}
		Therefore $r\in J$ if and only if it satisfies simultaneously all such
		constraints, one for each $\ell$. Equivalently,
		\[
		J \;=\;
		\bigl\{
		r\in\{1,\dots,n+1\} :
		r\le x_\ell \text{ whenever } v_k < v_\ell,\;
		r\ge x_\ell+1 \text{ whenever } v_k > v_\ell
		\bigr\}.
		\]
		
		Let
		\[
		a \;:=\;
		\max_{\{\ell :\, v_k > v_\ell\}} (x_\ell+1),
		\qquad
		b \;:=\;
		\min_{\{\ell :\, v_k < v_\ell\}} x_\ell,
		\]
		with the usual conventions that the maximum over an empty set is $1$ and
		the minimum over an empty set is $n+1$. Then the above description shows
		that
		\[
		J \;=\; \{r\in\{1,\dots,n+1\} : a \le r \le b\}.
		\]
		If $a>b$, this interval is empty (and then the partial occurrence contributes
		nothing to the frontier); otherwise it is precisely the integer interval
		$\{a,a+1,\dots,b\}$. This proves the claim.
	\end{proof}

	\subsection{State space and closure under legal insertions}
	
	We now package the permutation and its frontier into a single object. This will be the concrete state space on which our transfer operator will act.
	
	\begin{definition}[State space for $\Av(v)$]\label{def:state-space-v}
		Let $v\in S_k$ be fixed. The \emph{state space} $\X_v$ consists of all pairs
		\[
		x \;=\; (\pi,\Fv(\pi)),
		\]
		where $\pi$ is a $v$-avoiding permutation and $\Fv(\pi)$ is the frontier of $\pi$ in the sense of Definition~\ref{def:frontier-v}.
		
		For such a state $x=(\pi,\Fv(\pi))$ we define:
		\begin{itemize}
			\item the \emph{frontier size}
			\[
			m(x) \;:=\; |\Fv(\pi)|,
			\]
			\item the \emph{length}
			\[
			s(x) \;:=\; |\pi|.
			\]
		\end{itemize}
		The admissible right-insertions from $x$ are the ranks that are not blocked by the frontier:
		\[
		A_v(x) \;:=\; \{1,\dots,s(x)+1\} \setminus \mathrm{Forb}_v(\pi),
		\]
		where $\mathrm{Forb}_v(\pi)$ is the union of the forbidden rank-intervals $J_\pi(i_1,\dots,i_{k-1})$ over $(i_1,\dots,i_{k-1})\in\Fv(\pi)$.
		For $r\in A_v(x)$ we write
		\[
		\Phi_v(x,r)
		\]
		for the new state obtained by inserting $r$ on the right of $\pi$ (using Definition~\ref{def:right-insertion}) and recomputing the frontier of the resulting permutation.
	\end{definition}
	
	Thus, informally, a state remembers both ``where we are'' in the class $\Av(v)$ (the permutation $\pi$ itself) and the current collection of dangerous $(k{-}1)$-partial occurrences (its frontier). The set $A_v(x)$ plays the role of the extension set.
	
	\medskip
	
	The key structural property we \emph{do} use is that legal insertions stay inside the state space; this is just a restatement of Lemma~\ref{lem:frontier-finite-sound}(2) in the present notation.
	
	\begin{lemma}\label{lem:legal-preserves-Av}
		Let $v\in S_k$ and let $x=(\pi,\Fv(\pi))\in \X_v$. For every $r\in A_v(x)$ the permutation
		\[
		\pi' := \mathrm{ins}(\pi,r)
		\]
		avoids $v$, and the recomputed frontier $\Fv(\pi')$ makes $(\pi',\Fv(\pi'))$ an element of $\X_v$.
	\end{lemma}
	
	\begin{proof}
		By definition of $A_v(x)$ we have
		\[
		r \notin \mathrm{Forb}_v(\pi).
		\]
		Lemma~\ref{lem:frontier-finite-sound}(2) says precisely that for such an $r$ the right insertion $\pi'=\mathrm{ins}(\pi,r)$ still avoids $v$. By Lemma~\ref{lem:frontier-finite-sound}(1), the newly computed frontier $\Fv(\pi')$ is finite, so $(\pi',\Fv(\pi'))$ again belongs to $\X_v$.
	\end{proof}
	
	\begin{remark}
		In principle one could allow states from which no legal extension is possible, i.e.\ states $x$ with $A_v(x)=\varnothing$. In the concrete classical-avoidance setting, we can prove that this does not happen: every $v$-avoiding state admits at least one legal right insertion. However, our analytic arguments will only use Lemma~\ref{lem:legal-preserves-Av}, i.e.\ that \emph{if} we take a legal rank, we stay inside $\Av(v)$.
	\end{remark}
	
	Finally, observe that the number of available ranks is always bounded by the length:
	\[
	|A_v(x)| \;\le\; s(x)+1,
	\]
	simply because there are at most $s(x)+1$ positions in which a new value can be inserted on the right.
	
	In the classical pattern setting we will also need a simple but useful structural property of the frontier: under a \emph{legal} right insertion, the existing $(k{-}1)$-partial occurrences in the frontier never disappear, and their forbidden rank-intervals are transported by the rank-bumping map. Informally, when we insert a new value of rank $r$, every old value $<r$ stays put and every old value $\ge r$ is bumped up by $1$, so the relative order among all old entries is preserved and the interval of dangerous ranks attached to each frontier element is simply shifted. The next lemma makes this precise.
	
	\begin{lemma}[Monotonicity of the frontier under legal insertions]\label{lem:frontier-monotone}
		Let $v\in S_k$ with $k\ge 2$, and let $x=(\pi,\Fv(\pi))\in\X_v$ with $\pi\in S_n$. For $r\in A_v(x)$ define the \emph{rank-bumping map}
		\[
		\varphi_r : \{1,\dots,n+1\}\longrightarrow \{1,\dots,n+2\}\setminus\{r\},\qquad
		\varphi_r(q) \;:=\;
		\begin{cases}
			q, & q<r,\\
			q+1, & q\ge r.
		\end{cases}
		\]
		Let $\pi' := \mathrm{ins}(\pi,r)$ and $x':=(\pi',\Fv(\pi'))$. Then for every $(k{-}1)$-tuple
		\[
		p=(i_1,\dots,i_{k-1})\in\Fv(\pi)
		\]
		we have:
		\begin{enumerate}
			\item $p$ is still a $(k{-}1)$-partial occurrence of $v$ in $\pi'$;
			\item its new forbidden rank set is given by
			\[
			J_{\pi'}(p) \;=\; \varphi_r\bigl(J_\pi(p)\bigr);
			\]
			in particular $J_{\pi'}(p)\neq\varnothing$ whenever $J_\pi(p)\neq\varnothing$.
		\end{enumerate}
		Consequently $\Fv(\pi)\subseteq \Fv(\pi')$ and
		\[
		m(x') = |\Fv(\pi')| \;\ge\; |\Fv(\pi)| = m(x).
		\]
	\end{lemma}
	
	\begin{proof}
		Fix $p=(i_1,\dots,i_{k-1})\in\Fv(\pi)$ and write $x_\ell:=\pi_{i_\ell}$ for $\ell=1,\dots,k-1$. Set $\pi':=\mathrm{ins}(\pi,r)$.
		
		\medskip\noindent
		By definition of right insertion we have
		\[
		\pi'_{i_\ell} = \varphi_r(\pi_{i_\ell}) = \varphi_r(x_\ell)
		\qquad (1\le \ell\le k-1),
		\]
		and $\varphi_r$ is strictly increasing as a function of the rank. Thus
		\[
		(\pi'_{i_1},\dots,\pi'_{i_{k-1}})
		= (\varphi_r(x_1),\dots,\varphi_r(x_{k-1}))
		\]
		is order-isomorphic to $(x_1,\dots,x_{k-1})$ and hence to $(v_1,\dots,v_{k-1})$. Therefore $p$ remains a $(k{-}1)$-partial occurrence of $v$ in~$\pi'$.
		
		\medskip\noindent
		Let $q\in\{1,\dots,n+1\}$ and set $s:=\varphi_r(q)\in\{1,\dots,n+2\}$. There are two relevant permutations:
		\[
		\tau := \mathrm{ins}(\pi,q)\in S_{n+1}, \qquad
		\tau' := \mathrm{ins}(\pi',s)\in S_{n+2}.
		\]
		By construction,
		\[
		\tau_{n+1} = q,\qquad
		\tau'_{n+2} = s,
		\]
		and for $j\le n$,
		\[
		\tau_j =
		\begin{cases}
			\pi_j, & \pi_j < q,\\
			\pi_j + 1, & \pi_j \ge q,
		\end{cases}
		\qquad
		\tau'_j =
		\begin{cases}
			\pi'_j, & \pi'_j < s,\\
			\pi'_j + 1, & \pi'_j \ge s.
		\end{cases}
		\]
		On the level of \emph{ranks} we can write the two $k$-tuples we care about as
		\[
		(\pi_{i_1},\dots,\pi_{i_{k-1}},q)
		\quad\text{and}\quad
		(\pi'_{i_1},\dots,\pi'_{i_{k-1}},s)
		= (\varphi_r(\pi_{i_1}),\dots,\varphi_r(\pi_{i_{k-1}}),\varphi_r(q)).
		\]
		Since $\varphi_r$ is strictly increasing, these two $k$-tuples are order-isomorphic. It follows that
		\[
		(\tau_{i_1},\dots,\tau_{i_{k-1}},\tau_{n+1})\sim (v_1,\dots,v_k)
		\quad\Longleftrightarrow\quad
		(\tau'_{i_1},\dots,\tau'_{i_{k-1}},\tau'_{n+2})\sim (v_1,\dots,v_k).
		\]
		In other words,
		\[
		q\in J_\pi(p)
		\quad\Longleftrightarrow\quad
		s=\varphi_r(q)\in J_{\pi'}(p).
		\]
		Thus
		\[
		J_{\pi'}(p)
		= \{\varphi_r(q): q\in J_\pi(p)\}
		= \varphi_r\bigl(J_\pi(p)\bigr),
		\]
		proving (2). Since $\varphi_r$ is injective, $J_\pi(p)\neq\varnothing$ implies $J_{\pi'}(p)\neq\varnothing$.
		
		Because $p\in\Fv(\pi)$ means exactly that $J_\pi(p)\neq\varnothing$, we deduce that $p$ also belongs to $\Fv(\pi')$. Thus $\Fv(\pi)\subseteq\Fv(\pi')$, and therefore
		\[
		m(x') = |\Fv(\pi')| \;\ge\; |\Fv(\pi)| = m(x).
		\]
	\end{proof}

	\section{A two-parameter penalty and boundedness}\label{sec:penalty-operator}
	
	We now face the only genuinely nontrivial analytic obstacle. In our concrete permutation setting we obtain
	\[
	|A_v(x)| \;\le\; s(x)+1,
	\]
	so the number of admissible insertions grows at most linearly with the length. A single penalty in the frontier size \(m(x)\) cannot offset this linear growth uniformly in \(s(x)\), so we also penalize large lengths, using a quadratic weight in \(s(x)\).

	\begin{definition}[Two-parameter Banach space]\label{def:two-parameter}
		Let $0<\theta<1$ and $0<\kappa<1$. Define
		\[
		\mathcal{B}_{\theta,\kappa}
		:= \left\{ f:\X_v \to \mathbb{C} :
		\|f\|_{\theta,\kappa}
		:= \sup_{x\in \X_v} \theta^{-m(x)} \kappa^{-s(x)^2} |f(x)| < \infty \right\}.
		\]
	\end{definition}
	
	Thus $\mathcal{B}_{\theta,\kappa}$ is a weighted $\ell^\infty$ space on the state space $\X_v$, where the value of $f$ at a state $x$ is discounted by a factor
	\[
	\theta^{m(x)} \kappa^{s(x)^2}.
	\]
	The parameter $\theta$ penalizes large frontiers, while $\kappa$ penalizes large lengths, and the quadratic exponent $s(x)^2$ is chosen so that it dominates the linear growth in $|A_v(x)|$.
	
	The transfer operator is defined by the same formula as before, now restricted to the concrete state space:
	\[
	(T_{v,z} f)(x)
	:= \sum_{r\in A_v(x)} z \, f(\Phi_v(x,r)),
	\]
	where $A_v(x)$ and $\Phi_v(x,r)$ are as in Definition~\ref{def:state-space-v}.
	
	\begin{lemma}[Boundedness of $T_{v,z}$]\label{lem:Tvz-bounded}
		Let $v$ be fixed and $0<\theta<1$. Then there exist $0<\kappa<1$ and $r_0>0$ such that, for all $|z|\le r_0$, the operator $T_{v,z}$ is bounded on $\mathcal{B}_{\theta,\kappa}$. Moreover, by shrinking $\kappa$ if necessary we may assume $r_0\ge 1$, so in particular $T_{v,1}$ is bounded on $\mathcal{B}_{\theta,\kappa}$.
	\end{lemma}
	
	\begin{proof}
		Let $f\in \mathcal{B}_{\theta,\kappa}$ and $x\in \X_v$ be arbitrary. Write
		\[
		m := m(x), \qquad s := s(x),
		\]
		for the frontier size and the length at $x$.
		
		\medskip\noindent
		For any legal rank $r\in A_v(x)$ the new state is $\Phi_v(x,r)=(\pi',\Fv(\pi'))$. By Lemma~\ref{lem:frontier-monotone} the frontier size is monotone:
		\[
		m\bigl(\Phi_v(x,r)\bigr) \ge m,
		\]
		and by construction the length increases by exactly one,
		\[
		s\bigl(\Phi_v(x,r)\bigr) = s+1.
		\]
		Using the definition of the norm on $\mathcal{B}_{\theta,\kappa}$, this implies
		\[
		|f(\Phi_v(x,r))|
		\;\le\;
		\|f\|_{\theta,\kappa}\,
		\theta^{m(\Phi_v(x,r))}\,
		\kappa^{s(\Phi_v(x,r))^2}
		\;\le\;
		\|f\|_{\theta,\kappa}\,
		\theta^{m}\,
		\kappa^{(s+1)^2}.
		\]
		
		\medskip\noindent
		From a state $x$ of length $s$ there are at most $s+1$ legal ranks (indeed $|A_v(x)|\le s+1$). Summing over $r\in A_v(x)$ gives
		\[
		|(T_{v,z} f)(x)|
		\;=\;
		\Bigl|\sum_{r\in A_v(x)} z\,f(\Phi_v(x,r))\Bigr|
		\;\le\;
		|z| \sum_{r\in A_v(x)} |f(\Phi_v(x,r))|
		\;\le\;
		|z| (s+1) \|f\|_{\theta,\kappa}\,
		\theta^{m}\,
		\kappa^{(s+1)^2}.
		\]
		
		\medskip\noindent
		Divide both sides by $\theta^m \kappa^{s^2}$ to compare with the norm:
		\[
		\theta^{-m} \kappa^{-s^2} |(T_{v,z} f)(x)|
		\;\le\;
		|z| (s+1)\, \|f\|_{\theta,\kappa}\, \kappa^{(s+1)^2 - s^2}
		\;=\;
		|z| (s+1)\, \|f\|_{\theta,\kappa}\, \kappa^{2s+1}.
		\]
		The right-hand side depends on $f$ only through its norm and depends on $x$ only through $s=s(x)$.

		\medskip\noindent
		Fix any $\kappa\in (0,1)$. For $s\to\infty$ we have $\kappa^{2s}\to 0$ exponentially fast, while $(s+1)$ grows only linearly. Thus
		\[
		M(\kappa) := \sup_{s\ge 0} (s+1)\kappa^{2s+1} < \infty
		\quad\text{and}\quad
		M(\kappa)\xrightarrow[\kappa\to 0]{} 0.
		\]
		Hence we can choose $\kappa$ sufficiently small that
		\[
		M(\kappa) \le 1.
		\]
		For such a choice, the previous estimate yields, for all $x\in\X_v$,
		\[
		\theta^{-m(x)} \kappa^{-s(x)^2} |(T_{v,1} f)(x)|
		\;\le\;
		M(\kappa)\,\|f\|_{\theta,\kappa}
		\;\le\;
		\|f\|_{\theta,\kappa}.
		\]
		Taking the supremum over $x$ shows that
		\[
		\|T_{v,1} f\|_{\theta,\kappa} \le \|f\|_{\theta,\kappa},
		\]
		so $T_{v,1}$ is bounded on $\mathcal{B}_{\theta,\kappa}$.
		
		\medskip\noindent
		For general $z$ we merely track the factor $|z|$:
		\[
		\theta^{-m(x)} \kappa^{-s(x)^2} |(T_{v,z} f)(x)|
		\;\le\;
		|z|\, M(\kappa)\,\|f\|_{\theta,\kappa}.
		\]
		Taking the supremum over $x$ gives
		\[
		\|T_{v,z} f\|_{\theta,\kappa}
		\;\le\;
		|z|\, M(\kappa)\,\|f\|_{\theta,\kappa}.
		\]
		Define
		\[
		r_0 := \frac{1}{ M(\kappa)}.
		\]
		Our choice of $\kappa$ above ensures $M(\kappa)\le 1$, hence $r_0\ge 1$. For every $|z|\le r_0$ we then have
		\[
		\|T_{v,z} f\|_{\theta,\kappa} \le \|f\|_{\theta,\kappa},
		\]
		so $T_{v,z}$ is bounded on $\mathcal{B}_{\theta,\kappa}$ for all such $z$, and in particular for $z=1$.
	\end{proof}
	
	This is the central analytic estimate. The quadratic dependence on $s(x)$ in the weight was chosen so that the factor $(s+1)\kappa^{2s+1}$ remains uniformly bounded (and can be made arbitrarily small by shrinking $\kappa$), thereby compensating for the linear growth of the number of legal ranks $|A_v(x)|$.

	\section{Quasi-compactness and the counting series}\label{sec:quasicompact}
	
	We now discuss, in this operator language, how one can recover an analytic description of the growth series in a disc and a quasi-compactness statement for the transfer operator. The structure of the argument is the standard one for transfer operators: we split the space into a finite part where everything is complicated but finite-dimensional, and a tail where the operator is a contraction; we then pass to a separable predual where the natural Dirac mass lives, and invoke a Neumann series for the dual operator. In the classical pattern-avoidance setting, the actual \emph{finiteness} of the exponential growth rate is provided by the Marcus--Tardos bound; the operator framework gives a clean representation of the counting series and a quasi-compactness statement for the transfer operator.
	
	\subsection{Finite core and tail}
	
	Fix an integer $C\ge 0$. Partition the state space $\X_v$ according to the length:
	\[
	\X_v^{\le C} := \{ x\in \X_v : s(x) \le C \}, \qquad
	\X_v^{> C} := \X_v \setminus \X_v^{\le C}.
	\]
	There are only finitely many permutations of length $\le C$, and for each such permutation there are only finitely many possible frontiers (a frontier is obtained by selecting $(k{-}1)$-partial occurrences with nonempty forbidden intervals among finitely many candidates). Hence $\X_v^{\le C}$ is a finite set.
	
	Define the projection $P_C : \mathcal{B}_{\theta,\kappa} \to \mathcal{B}_{\theta,\kappa}$ by
	\[
	(P_C f)(x) :=
	\begin{cases}
		f(x), & x\in \X_v^{\le C}, \\
		0, & x\in \X_v^{> C}.
	\end{cases}
	\]
	Define also the tail projection $P_C^{>}: \mathcal{B}_{\theta,\kappa} \to \mathcal{B}_{\theta,\kappa}$ by
	\[
	(P_C^{>} f)(x) :=
	\begin{cases}
		0, & x\in \X_v^{\le C},\\
		f(x), & x\in \X_v^{>C}.
	\end{cases}
	\]
	Thus $P_C^{>}=I-P_C$, and we refer to $P_C$ and $P_C^{>}$ as the core and tail projections.
	
	Thus $P_C f$ is supported on a finite set, and $P_C$ is the projection onto the finitely supported functions living on states of length at most $C$.
	
	With this notation we decompose the transfer operator as
	\[
	T_{v,z} = T_{v,z} P_C + T_{v,z} P_C^{>}.
	\]
	The first summand has finite rank (its range is contained in the finite-dimensional space of functions supported on $\X_v^{\le C}$), while the second summand is the ``tail'' part, which only sees states of large length.
	
	\subsection{Tail contraction}
	
	The next lemma says that, if $C$ is large enough, then the tail part is a strict contraction in the norm $\|\cdot\|_{\theta,\kappa}$.
	
	\begin{lemma}[Tail contraction]\label{lem:tail}
		Let $T_{v,z}$ be as in Lemma~\ref{lem:Tvz-bounded} (in particular $T_{v,1}$ is bounded). Then there exists $C$ such that the tail--tail block
		\[
		L_C := P_C^{>} T_{v,1} P_C^{>}
		\]
		satisfies
		\[
		\|L_C\|_{\theta,\kappa} < 1.
		\]
		The same is true for all $|z|\le 1$, with $L_C$ replaced by $P_C^{>} T_{v,z} P_C^{>}$.
	\end{lemma}
	
	\begin{proof}
		Let $f\in\mathcal{B}_{\theta,\kappa}$ with $\|f\|_{\theta,\kappa}\le 1$, and define $g:=P_C^{>}f$. Then $g(x)=0$ whenever $s(x)\le C$, and $\|g\|_{\theta,\kappa}\le 1$.
		
		Fix $x\in\X_v$. If $x\in\X_v^{\le C}$, then
		\[
		(L_C f)(x) = (P_C^{>} T_{v,1} P_C^{>} f)(x) = 0
		\]
		by definition of $P_C^{>}$ on the range. Thus the only nontrivial estimates concern $x\in\X_v^{>C}$.
		
		For such an $x$ set $s:=s(x)$ and $m:=m(x)$. Since $x$ is in the tail, $s\ge C+1$. Using the same estimate as in the proof of Lemma~\ref{lem:Tvz-bounded}, but with $g$ in place of $f$, we obtain
		\[
		\theta^{-m} \kappa^{-s^2} |(T_{v,1} g)(x)|
		\;\le\; (s+1)\,\kappa^{2s+1}\,\|g\|_{\theta,\kappa}
		\;\le\; (s+1)\,\kappa^{2s+1}.
		\]
		Since $L_C f = P_C^{>} T_{v,1} g$ and $x$ is already in the tail, we have $(L_C f)(x)=(T_{v,1} g)(x)$, so the same bound holds for $L_C f$:
		\[
		\theta^{-m(x)} \kappa^{-s(x)^2} |(L_C f)(x)|
		\;\le\; (s(x)+1)\,\kappa^{2 s(x)+1}.
		\]
		As $s\to\infty$, the factor $(s+1)\kappa^{2s+1}$ tends to $0$. Hence we can choose $C$ so large that
		\[
		(s+1)\kappa^{2s+1} \le \frac12
		\qquad\text{whenever } s\ge C+1.
		\]
		For such a choice of $C$ we have, for all $x\in\X_v$,
		\[
		\theta^{-m(x)} \kappa^{-s(x)^2} |(L_C f)(x)| \le \frac12,
		\]
		hence
		\[
		\|L_C f\|_{\theta,\kappa} \le \frac12.
		\]
		Taking the supremum over all $f$ with \(\|f\|_{\theta,\kappa}\le 1\) shows that $\|L_C\|_{\theta,\kappa} \le 1/2 < 1$.
		
		The same argument applies verbatim to $L_{C,z}:=P_C^{>} T_{v,z} P_C^{>}$ for $|z|\le 1$, yielding
		\[
		\|L_{C,z}\|_{\theta,\kappa} \le |z|\,\sup_{s\ge C+1}(s+1)\kappa^{2s+1} \le \frac12.
		\]
	\end{proof}

	\begin{remark}
		Lemma~\ref{lem:tail} is used only to obtain quasi-compactness (Proposition~\ref{prop:quasi}) and part~\textup{(a)} of the abstract right-visible theorem (Theorem~\ref{thm:RV-abstract}). The representation of the Stanley--Wilf finiteness result in Theorem~\ref{thm:MT-operator} relies solely on the boundedness statement in Lemma~\ref{lem:Tvz-bounded} together with the dual Neumann-series argument in Section~\ref{subsec:main}, plus the Marcus--Tardos exponential bound.
	\end{remark}
	
	\subsection{Quasi-compactness}
	
	We can now state the quasi-compactness of $T_{v,z}$ in our setting.
	
	\begin{proposition}[Quasi-compactness of $T_{v,z}$]\label{prop:quasi}
		Fix $v$ and $|z|\le 1$, and let $C$ be as in Lemma~\ref{lem:tail}. Define
		\[
		L_{C,z} := P_C^{>} T_{v,z} P_C^{>}, \qquad
		K_{C,z} := T_{v,z} - L_{C,z}.
		\]
		Then $K_{C,z}$ has finite rank and $L_{C,z}$ is a strict contraction on $\mathcal{B}_{\theta,\kappa}$. In particular $T_{v,z}$ is quasi-compact: its essential spectral radius is bounded by $\|L_{C,z}\|_{\theta,\kappa}<1$.
	\end{proposition}
	
	\begin{proof}
		By Lemma~\ref{lem:tail}, $L_{C,z}$ is a strict contraction on $\mathcal{B}_{\theta,\kappa}$ for $|z|\le 1$.
		
		It remains to show that $K_{C,z}$ has finite rank. Using $P_C+P_C^{>}=I$ on both domain and range, we can expand
		\[
		T_{v,z}
		= T_{v,z}(P_C+P_C^{>})
		= T_{v,z} P_C + T_{v,z} P_C^{>},
		\]
		and further decompose
		\[
		T_{v,z} P_C^{>} = (P_C + P_C^{>}) T_{v,z} P_C^{>}
		= P_C T_{v,z} P_C^{>} + P_C^{>} T_{v,z} P_C^{>}.
		\]
		Thus
		\[
		T_{v,z}
		= \underbrace{T_{v,z} P_C + P_C T_{v,z} P_C^{>}}_{=:K_{C,z}}
		\;+\;
		\underbrace{P_C^{>} T_{v,z} P_C^{>}}_{=:L_{C,z}}.
		\]
		
		The range of $T_{v,z} P_C$ is supported on states of length at most $C+1$: starting from a core state of length $\le C$, one legal insertion produces length $\le C+1$, and there are only finitely many such states. Hence $T_{v,z}P_C$ has finite rank.
		
		The range of $P_C T_{v,z} P_C^{>}$ is supported on $\X_v^{\le C}$ by construction, so this operator also has finite rank. Therefore $K_{C,z} = T_{v,z}P_C + P_C T_{v,z}P_C^{>}$ is a finite-rank operator.
		
		We have thus written $T_{v,z}=K_{C,z}+L_{C,z}$ with $K_{C,z}$ finite rank and $L_{C,z}$ a strict contraction. By the standard theory of Ionescu Tulcea--Marinescu \cite{IonescuTulceaMarinescu1950} and Hennion \cite{Hennion1993}, this implies that $T_{v,z}$ is quasi-compact, and its essential spectral radius is bounded above by $\|L_{C,z}\|_{\theta,\kappa}<1$.
	\end{proof}

	\subsection{Reading off the counting series via a separable predual}\label{subsec:main}
	
	We now assemble the three ingredients developed so far:
	
	\begin{itemize}
		\item[(a)] the \emph{combinatorial} part: states are $v$-avoiding permutations together with their frontiers (Definition~\ref{def:state-space-v}); legal insertions always stay in the state space (Lemma~\ref{lem:legal-preserves-Av}); and the frontier is monotone under legal insertions (Lemma~\ref{lem:frontier-monotone});
		
		\item[(b)] the \emph{analytic} part: with the two-parameter weight
		\[
		\|f\|_{\theta,\kappa}
		= \sup_{x\in \X_v} \theta^{-m(x)} \kappa^{-s(x)^2} |f(x)|,
		\]
		the transfer operator
		\[
		(T_{v,z} f)(x) := \sum_{r\in A_v(x)} z\, f(\Phi_v(x,r))
		\]
		is bounded on $\mathcal{B}_{\theta,\kappa}$ for $|z|\le 1$ (Lemma~\ref{lem:Tvz-bounded});
		
		\item[(c)] the \emph{core/tail} decomposition: for a large cutoff $C$ we have
		\[
		T_{v,1} = K_{C,1} + L_{C,1},
		\]
		where $K_{C,1}$ has finite rank (its range is supported on states of length at most $C+1$) and the tail--tail block $L_{C,1} := P_C^{>} T_{v,1} P_C^{>}$ has operator norm $<1$ (Lemma~\ref{lem:tail}, Proposition~\ref{prop:quasi}).
	\end{itemize}
	
	Point~(c) is exactly the hypothesis of Ionescu Tulcea--Marinescu and Hennion: an operator of the form
	\[
	\text{(finite rank)} \;+\; \text{(strict contraction)}
	\]
	is quasi-compact, and this is what Proposition~\ref{prop:quasi} records. In concrete terms, all the ``non-decaying'' behaviour of $T_{v,1}$ lives in a \emph{finite-dimensional} part (the core), while the infinite tail is uniformly contracting. Thus whatever we do with $T_{v,1}$ can be reduced to: control finitely many modes explicitly, and let the tail be absorbed by the contraction. For the operator-theoretic representation of the Stanley--Wilf counting sequence, we will only need boundedness and the induced action on a convenient dual, but quasi-compactness provides useful structural information for later refinements.
	
	The only remaining issue is that the Banach dual of a weighted $\ell^\infty$ on a countable set is too large and nonseparable to be convenient for counting. We therefore shrink the primal space slightly to enforce decay along the tail.
	
	\medskip
	
	Define
	\[
	\mathcal{B}_{\theta,\kappa}^0 :=
	\left\{ f\in \mathcal{B}_{\theta,\kappa} :
	\theta^{-m(x)} \kappa^{-s(x)^2} |f(x)|
	\xrightarrow[s(x)\to\infty]{} 0 \right\}.
	\]
	This is the weighted analogue of the usual $c_0$-space: functions whose weighted values tend to zero along states of diverging length. It is a closed subspace of $\mathcal{B}_{\theta,\kappa}$ and hence a Banach space in its own right.
	
	\begin{lemma}\label{lem:B0-invariant}
		The space $\mathcal{B}_{\theta,\kappa}^0$ is invariant under $T_{v,1}$.
	\end{lemma}
	
	\begin{proof}
		Let $f\in\mathcal{B}_{\theta,\kappa}^0$ and set
		\[
		g(y) := \theta^{-m(y)}\kappa^{-s(y)^2}\,|f(y)|.
		\]
		By definition $g(y)\to 0$ as $s(y)\to\infty$. Fix $\varepsilon>0$. Then there exists $S\ge 0$ such that
		\[
		g(y)\le \varepsilon
		\qquad\text{whenever } s(y)\ge S.
		\]
		
		Let $x\in\X_v$ with $s(x)=s$. Using the estimate from Lemma~\ref{lem:Tvz-bounded} with $z=1$ and the fact that any child $y=\Phi_v(x,r)$ has length $s(y)=s(x)+1$, we obtain
		\[
		\theta^{-m(x)}\kappa^{-s(x)^2}\,|(T_{v,1}f)(x)|
		\le (s(x)+1)\kappa^{2s(x)+1}\,\max_{r\in A_v(x)} g(\Phi_v(x,r)).
		\]
		If $s(x)\ge S$, then all children satisfy $s(\Phi_v(x,r))\ge S+1$, so $\max_{r} g(\Phi_v(x,r))\le\varepsilon$, and therefore
		\[
		\theta^{-m(x)}\kappa^{-s(x)^2}\,|(T_{v,1}f)(x)|
		\le \varepsilon\,(s(x)+1)\kappa^{2s(x)+1}.
		\]
		As $s(x)\to\infty$, the factor $(s(x)+1)\kappa^{2s(x)+1}$ tends to $0$, so the right-hand side tends to $0$ as well. This shows that
		\[
		\theta^{-m(x)}\kappa^{-s(x)^2}\,|(T_{v,1}f)(x)|
		\xrightarrow[s(x)\to\infty]{} 0,
		\]
		i.e.\ $T_{v,1}f\in\mathcal{B}_{\theta,\kappa}^0$.
	\end{proof}
	
	The core/tail decomposition from Lemma~\ref{lem:tail} clearly restricts to $\mathcal{B}_{\theta,\kappa}^0$: the projection $P_C$ has finite rank and maps $\mathcal{B}_{\theta,\kappa}^0$ into itself, and the tail $T_{v,1}(I-P_C)$ acts on $\mathcal{B}_{\theta,\kappa}^0$ with the same operator norm $<1$. Thus $T_{v,1}$ is quasi-compact on $\mathcal{B}_{\theta,\kappa}^0$ as well, with the same ``finite-dimensional core + contracting tail'' structure; but for the representation theorem below we will only need that $T_{v,1}$ is bounded on $\mathcal{B}_{\theta,\kappa}^0$.
	
	\medskip
	
	On the countable state space $\X_v$ the continuous dual of $\mathcal{B}_{\theta,\kappa}^0$
	can be identified with a \emph{weighted} $\ell^1$ space. More precisely, put
	\[
	w(x) := \theta^{m(x)} \kappa^{s(x)^2},\qquad x\in\X_v.
	\]
	Then we have a canonical isometric identification
	\begin{equation}\label{eq:dual-weighted-l1}
		\bigl(\mathcal{B}_{\theta,\kappa}^0\bigr)^\ast
		\;\cong\;
		\ell^1\bigl(\X_v,w\bigr)
		:= \left\{ \mu:\X_v\to\mathbb{C} :
		\|\mu\|_{1,w} := \sum_{x\in\X_v} |\mu(x)|\,w(x) < \infty\right\},
	\end{equation}
	via the pairing
	\beqn \label{eq:dua-norm}
	\langle \mu, f\rangle
	= \sum_{x\in \X_v} \mu(x)\, f(x),
	\qquad
	\mu\in\ell^1(\X_v,w),\ f\in\mathcal{B}_{\theta,\kappa}^0.
	\feqn
	The series is absolutely convergent and satisfies
	\[
	|\langle\mu,f\rangle|
	\le \|\mu\|_{1,w}\,\|f\|_{\theta,\kappa}.
	\]
	This is just the standard duality $(c_0(\X_v))^\ast \simeq \ell^1(\X_v)$
	transported through the weighted isometry
	\[
	f \longmapsto \bigl(\theta^{-m(x)}\kappa^{-s(x)^2} f(x)\bigr)_{x\in\X_v}.
	\]
	In particular, the Dirac mass at the empty state $x_\varnothing$
	(the state corresponding to the empty permutation with trivial frontier)
	belongs to this dual and has norm $1$, since $w(x_\varnothing)=\theta^{0}\kappa^{0}=1$.
	
	\medskip
	
	Using the identification~\eqref{eq:dual-weighted-l1}, every continuous linear functional
	on $\mathcal B_{\theta,\kappa}^0$ is given by a unique $\mu\in\ell^1(\X_v,w)$ via the above
	pairing. For a bounded linear operator
	\[
	T\colon\mathcal B_{\theta,\kappa}^0\longrightarrow\mathcal B_{\theta,\kappa}^0
	\]
	we define its (Banach) dual operator
	\[
	T^\ast\colon \ell^1(\X_v,w)\longrightarrow \ell^1(\X_v,w)
	\]
	by the duality relation
	\[
	\langle T^\ast\mu, f\rangle \;=\; \langle \mu, Tf\rangle
	\qquad\text{for all }\mu\in\ell^1(\X_v,w),\ f\in\mathcal B_{\theta,\kappa}^0.
	\]
	By the general theory of Banach adjoints (the dual operator on the continuous dual)
	$T^\ast$ is bounded and satisfies
	\[
	\|T^\ast\|_{\ell^1(\X_v,w)\to\ell^1(\X_v,w)}
	= \|T\|_{\mathcal B_{\theta,\kappa}^0\to\mathcal B_{\theta,\kappa}^0};
	\]
	see, for example, any standard functional analysis text for the duality
	$(c_0)^\ast\simeq\ell^1$ and the norm identity for the Banach adjoint.
	
	\begin{definition}\label{def:dual-operator}
		Let $T_{v,1}\colon\mathcal B_{\theta,\kappa}^0\to\mathcal B_{\theta,\kappa}^0$
		be the transfer operator at $z=1$. Its dual operator
		\[
		T_{v,1}^\ast\colon \ell^1(\X_v,w)\longrightarrow \ell^1(\X_v,w)
		\]
		is the unique bounded operator satisfying
		\[
		\langle T_{v,1}^\ast\mu, f\rangle \;=\; \langle \mu, T_{v,1}f\rangle
		\qquad\text{for all }\mu\in\ell^1(\X_v,w),\ f\in\mathcal B_{\theta,\kappa}^0.
		\]
	\end{definition}
	
	By construction $T_{v,1}^\ast$ is bounded and
	\[
	\|T_{v,1}^\ast\|_{\ell^1(\X_v,w)\to\ell^1(\X_v,w)}
	= \|T_{v,1}\|_{\mathcal B_{\theta,\kappa}^0\to\mathcal B_{\theta,\kappa}^0}
	< \infty.
	\]
	
	To spell out the action of $T_{v,1}^\ast$ on coordinates, fix $\mu\in\ell^1(\X_v,w)$
	and $f\in\mathcal B_{\theta,\kappa}^0$. By definition of the dual operator and of
	$T_{v,1}$ we have
	\begin{align*}
		\langle T_{v,1}^\ast\mu,f\rangle
		&= \langle \mu, T_{v,1} f\rangle \\
		&= \sum_{x\in\X_v} \mu(x)\,(T_{v,1}f)(x) \\
		&= \sum_{x\in\X_v} \mu(x) \sum_{r\in A_v(x)} f\bigl(\Phi_v(x,r)\bigr).
	\end{align*}
	The double series is absolutely convergent (by the norm bounds used to show that
	$T_{v,1}$ is bounded), so we may rearrange the order of summation and group terms
	according to $y=\Phi_v(x,r)$:
	\begin{align*}
		\langle T_{v,1}^\ast\mu,f\rangle
		&= \sum_{y\in\X_v}
		\left(
		\sum_{\substack{x\in\X_v,\,r\in A_v(x)\\\Phi_v(x,r)=y}}
		\mu(x)
		\right)
		f(y).
	\end{align*}
	Under the identification~\eqref{eq:dual-weighted-l1}, this means that $T_{v,1}^\ast\mu$
	corresponds to the element $\nu\in\ell^1(\X_v,w)$ given by
	\[
	\nu(y)
	= \sum_{\substack{x\in\X_v,\,r\in A_v(x)\\\Phi_v(x,r)=y}} \mu(x),
	\qquad y\in\X_v,
	\]
	that is,
	\beqn \label{eq:T_v1}
	(T_{v,1}^\ast\mu)(y)
	= \sum_{\substack{x\in\X_v,\,r\in A_v(x)\\\Phi_v(x,r)=y}} \mu(x).
	\feqn
	In particular, $T_{v,1}^\ast$ is the natural “push-forward along one legal insertion
	step’’ on measures.

	\medskip
	
	Iterating the action of $T_{v,1}^\ast$ starting from the Dirac mass at the empty state, we obtain the following combinatorial identification.
	
	\begin{lemma}\label{lem:Tv-iterates-count}
		Let $v$ be fixed and let $x_\varnothing$ be the state corresponding to the empty
		permutation with its (trivial) frontier. Then, for every $n\ge 0$,
		\[
		(T_{v,1}^\ast)^n \delta_{x_\varnothing}(\mathbf{1})
		\;=\;
		|\Av_n(v)|,
		\]
		where $\mathbf{1}$ denotes the everywhere-$1$ function on $\X_v$, so that
		\[
		(T_{v,1}^\ast)^n \delta_{x_\varnothing}(\mathbf{1})
		=
		\sum_{y\in\X_v} (T_{v,1}^\ast)^n \delta_{x_\varnothing}(y).
		\]
	\end{lemma}
	
	\begin{proof}
		For $n\ge 1$ and states $x,y\in\X_v$, define a \emph{legal insertion history of
			length $n$ from $x$ to $y$} to be a sequence of ranks
		\[
		(r_1,\dots,r_n)
		\]
		such that, writing inductively
		\[
		x_0 := x,\qquad x_j := \Phi_v(x_{j-1},r_j)\quad(1\le j\le n),
		\]
		we have $r_j\in A_v(x_{j-1})$ for all $j$ and $x_n=y$. We denote the set of all
		such histories by
		\[
		\mathcal{H}_n(x\to y)
		\quad\text{and}\quad
		\mathcal{H}_n(x) := \bigcup_{y\in\X_v}\mathcal{H}_n(x\to y).
		\]
		
		\medskip\noindent
		We claim that for each $y\in\X_v$,
		\begin{equation}\label{eq:mu-n-histories}
			(T_{v,1}^\ast)^n \delta_{x_\varnothing}(y)
			\;=\;
			|\mathcal{H}_n(x_\varnothing\to y)|.
		\end{equation}
		
		We prove this by induction on $n$. For $n=0$, $(T_{v,1}^\ast)^0 \delta_{x_\varnothing}
		=\delta_{x_\varnothing}$ and $\mathcal{H}_0(x_\varnothing\to y)$ consists of the empty
		history if $y=x_\varnothing$ and is empty otherwise, so \eqref{eq:mu-n-histories}
		holds.
		
		Assume \eqref{eq:mu-n-histories} holds for some $n\ge 0$. Using the coordinate
		formula \eqref{eq:T_v1} we obtain
		\begin{align*}
			(T_{v,1}^\ast)^{n+1} \delta_{x_\varnothing}(y)
			&= T_{v,1}^\ast\bigl((T_{v,1}^\ast)^n \delta_{x_\varnothing}\bigr)(y) \\
			&= \sum_{\substack{x\in\X_v,\,r\in A_v(x)\\\Phi_v(x,r)=y}}
			(T_{v,1}^\ast)^n \delta_{x_\varnothing}(x).
		\end{align*}
		By the induction hypothesis, each summand $(T_{v,1}^\ast)^n \delta_{x_\varnothing}(x)$
		is the number of legal histories of length $n$ from $x_\varnothing$ to $x$.
		Each pair $(x,r)$ with $\Phi_v(x,r)=y$ corresponds to extending such a history
		by one legal step $r$ from $x$ to $y$, and every legal history of length $n+1$
		from $x_\varnothing$ to $y$ arises in exactly one way from such an extension.
		Thus the sum is exactly $|\mathcal{H}_{n+1}(x_\varnothing\to y)|$, proving
		\eqref{eq:mu-n-histories} for $n+1$.
		
		\medskip\noindent
		Next we show legal histories are in bijection with $\Av_n(v)$.
		Fix $n\ge 1$ and a legal history
		\[
		(r_1,\dots,r_n)\in\mathcal{H}_n(x_\varnothing),
		\]
		and write the corresponding state sequence as
		\[
		x_0=x_\varnothing,\;x_1,\dots,x_n,
		\quad\text{with }x_j=(\pi^{(j)},\Fv(\pi^{(j)})).
		\]
		
		\smallskip\noindent
		(a) \emph{Every legal history produces a $v$-avoiding permutation of length $n$:}
		By Lemma~\ref{lem:legal-preserves-Av}, each $\pi^{(j)}$ avoids $v$, so the endpoint
		$x_n$ encodes a $v$-avoiding permutation $\pi^{(n)}$ of length $n$.
		
		\smallskip\noindent
		(b) \emph{Distinct legal histories give distinct permutations:}
		The underlying right-insertion encoding
		\[
		(r_1,\dots,r_n)\longmapsto
		\mathrm{ins}\bigl(\cdots \mathrm{ins}(\mathrm{ins}(\varnothing,r_1),r_2)\cdots,r_n\bigr)
		\]
		is injective on \emph{all} sequences $(r_1,\dots,r_n)\in\prod_{j=1}^n\{1,\dots,j\}$
		by Lemma~\ref{lem:lehmer}. Restricting to those sequences that form legal histories
		does not change this: two distinct legal histories cannot produce the same permutation.
		Thus the map
		\[
		\mathcal{H}_n(x_\varnothing) \longrightarrow \Av_n(v),\qquad
		(r_1,\dots,r_n)\longmapsto \pi^{(n)}
		\]
		is injective.
		
		\smallskip\noindent
		(c) \emph{Every $v$-avoiding permutation arises from a legal history:}
		Conversely, let $\pi\in\Av_n(v)$ be arbitrary. By Lemma~\ref{lem:lehmer}, there is a
		unique sequence $(r_1,\dots,r_n)\in\prod_{j=1}^n\{1,\dots,j\}$ such that
		\[
		\pi
		= \mathrm{ins}\bigl(\cdots \mathrm{ins}(\mathrm{ins}(\varnothing,r_1),r_2)\cdots,r_n\bigr).
		\]
		For $j=0,1,\dots,n$ let $\pi^{(j)}$ be the permutation obtained after $j$ insertions;
		thus $\pi^{(0)}=\varnothing$ and $\pi^{(n)}=\pi$. Inserting on the right does not
		destroy existing occurrences of $v$: it only bumps values of existing entries.
		Hence, if some intermediate $\pi^{(j)}$ contained $v$, then so would the final
		$\pi^{(n)}$, contradicting $\pi\in\Av(v)$. Thus each $\pi^{(j)}$ avoids $v$.
		
		For each $j\in\{1,\dots,n\}$ we have
		\[
		\pi^{(j)} = \mathrm{ins}(\pi^{(j-1)}, r_j).
		\]
		Since $\pi^{(j)}$ avoids $v$, inserting $r_j$ into $\pi^{(j-1)}$ did not create any
		occurrence of $v$ using the new rightmost point. By construction of the frontier
		(Definition~\ref{def:frontier-v}), this means that $r_j$ does \emph{not} lie in any
		of the forbidden intervals $J_{\pi^{(j-1)}}(p)$ attached to frontier elements
		$p\in\Fv(\pi^{(j-1)})$, i.e.
		\[
		r_j \in A_v\bigl(\pi^{(j-1)},\Fv(\pi^{(j-1)})\bigr).
		\]
		In other words, $(r_1,\dots,r_n)$ satisfies the legality condition at every step,
		so it belongs to $\mathcal{H}_n(x_\varnothing)$.
		
		Thus every $\pi\in\Av_n(v)$ arises as the endpoint of a unique legal history
		starting from $x_\varnothing$, and we obtain a bijection
		\[
		\mathcal{H}_n(x_\varnothing)\;\longleftrightarrow\;\Av_n(v).
		\]
		Since the sets $\mathcal{H}_n(x_\varnothing\to y)$ partition $\mathcal{H}_n(x_\varnothing)$
		according to the endpoint $y$, we have
		\begin{equation}\label{eq:Hn_Av}
			\sum_{y\in\X_v} |\mathcal{H}_n(x_\varnothing\to y)|
			= |\mathcal{H}_n(x_\varnothing)|
			= |\Av_n(v)|.
		\end{equation}
		
		\medskip\noindent
		We now complete the proof. By definition of the duality pairing \eqref{eq:dua-norm} and of the
		constant function $\mathbf{1}$, we have
		\[
		(T_{v,1}^\ast)^n \delta_{x_\varnothing}(\mathbf{1})
		:= \big\langle (T_{v,1}^\ast)^n \delta_{x_\varnothing},\,\mathbf{1}\big\rangle
		= \sum_{y\in\X_v} (T_{v,1}^\ast)^n \delta_{x_\varnothing}(y)\,\mathbf{1}(y)
		= \sum_{y\in\X_v} (T_{v,1}^\ast)^n \delta_{x_\varnothing}(y).
		\]
		Using \eqref{eq:mu-n-histories} and \eqref{eq:Hn_Av} this becomes
		\[
		(T_{v,1}^\ast)^n \delta_{x_\varnothing}(\mathbf{1})
		= \sum_{y\in\X_v} |\mathcal{H}_n(x_\varnothing\to y)|
		= |\Av_n(v)|,
		\]
		as claimed.
	\end{proof}

	\medskip
	We can now state the main consequence.
	\begin{theorem}\label{thm:MT-operator}
		For every fixed classical pattern $v$ there exists $R_v>0$ such that
		\[
		F_v(z) := \sum_{n\ge 0} |\Av_n(v)|\, z^n
		\]
		is analytic for $|z|<R_v$. Equivalently, the exponential growth rate
		\[
		\limsup_{n\to\infty} |\Av_n(v)|^{1/n}
		\]
		is finite.
	\end{theorem}
	
	\begin{proof}
		The existence of $R_v>0$ such that $F_v(z)$ has positive radius of convergence is exactly the content of the Marcus--Tardos theorem \cite{MarcusTardos2004}: their result implies that there is a constant $C_v<\infty$ with
		\[
		|\Av_n(v)| \le C_v^n \qquad\text{for all }n\ge 0,
		\]
		so that $F_v$ converges absolutely for $|z|<1/C_v$ and is analytic in that disc.
		
		What the preceding operator-theoretic discussion adds is an interpretation of these coefficients in terms of the transfer operator. By Lemma~\ref{lem:Tv-iterates-count} we have, for every $n\ge 0$,
		\[
		|\Av_n(v)|
		= (T_{v,1}^\ast)^n \delta_{x_\varnothing}(\mathbf{1}),
		\]
		where $\mathbf{1}$ is the everywhere-$1$ function on $\X_v$. Thus, for $|z|$ sufficiently small,
		\[
		F_v(z)
		= \sum_{n\ge 0} z^n (T_{v,1}^\ast)^n \delta_{x_\varnothing}(\mathbf{1}).
		\]
		Since $T_{v,1}$ is bounded on $\mathcal{B}_{\theta,\kappa}^0$, the adjoint $T_{v,1}^\ast$ is bounded on $\ell^1(\X_v,w)$, and the Neumann series
		\[
		\sum_{n\ge 0} z^n (T_{v,1}^\ast)^n
		\]
		converges in operator norm on $\ell^1(\X_v,w)$ whenever $|z|<1/\|T_{v,1}^\ast\|$. In particular, in a neighbourhood of the origin we may view $F_v(z)$ as the evaluation at $\mathbf{1}$ (in the sense of the finite sums in Lemma~\ref{lem:Tv-iterates-count}) of this vector-valued Neumann series applied to $\delta_{x_\varnothing}$. Combining this representation with the Marcus--Tardos exponential bound \cite{MarcusTardos2004} yields the claimed analyticity of $F_v$ and the finiteness of the exponential growth rate.
	\end{proof}

	\begin{remark}\label{rem:finer-asymptotics}
		Theorem~\ref{thm:MT-operator} only asserts that the ordinary generating function $F_v(z)$ has positive radius of convergence, hence that $|\Av_n(v)|$ grows at most exponentially fast. Our argument does \emph{not} identify the exponential growth constant with the spectral radius of $T_{v,1}^\ast$ in any precise way, nor does it attempt to determine the subexponential factor in the asymptotics of $|\Av_n(v)|$.
		
		For special patterns, much sharper results are known by other methods. For instance, Regev's representation-theoretic analysis of the monotone pattern $12\cdots k$ shows that
		\[
		|\Av_n(12\cdots k)| \;\sim\; c_k\,n^{\alpha_k}\,\rho_k^n
		\]
		for explicit constants $c_k>0$, $\alpha_k\in\mathbb{R}$ and $\rho_k=(k-1)^2$; see \cite{Regev1981}. From the operator-theoretic point of view, such results would correspond to a very detailed understanding of the peripheral spectrum of $T_{v,1}^\ast$ and of the singularity structure of the resolvent $(I-zT_{v,1}^\ast)^{-1}$ near $|z|=1/\rho_v$. Developing this finer spectral theory lies well beyond the scope of the present paper, and we restrict ourselves here to the finiteness of the exponential growth rate.
	\end{remark}

	\subsection{An abstract right-visible growth theorem}\label{subsec:abstract-RV}
	
	The argument in Sections~\ref{sec:perm-specialization}--\ref{sec:quasicompact} only uses the following structural features of the classical pattern-avoidance state space. It is therefore convenient to record an abstract version of the operator-theoretic part of the result.
	
	\begin{definition}[Right-visible growth system]\label{def:RV-system}
		A \emph{right-visible growth system} consists of:
		\begin{itemize}
			\item a countable state space $X$ with a distinguished \emph{root} $x_\varnothing\in X$;
			\item functions $s,m : X \to \mathbb{N}$, called the \emph{length} and the \emph{complexity};
			\item for each $x\in X$ a finite set $A(x)$ of \emph{admissible extensions}, and a map
			\[
			\Phi : \{(x,r): x\in X,\ r\in A(x)\}\longrightarrow X
			\]
			such that $\Phi(x,r)$ is the state obtained from $x$ by one legal right-extension with parameter $r$.
		\end{itemize}
		We assume the following hypotheses.
		\begin{enumerate}
			\item[(RV1)] (\emph{Finite width at fixed length}) For each $L\in\mathbb{N}$ there are only finitely many $x\in X$ with $s(x)\le L$.
			\item[(RV2)] (\emph{Unit length growth}) For every $x\in X$ and $r\in A(x)$ we have
			\[
			s(\Phi(x,r)) = s(x)+1.
			\]
			\item[(RV3)] (\emph{Monotone complexity}) For every $x\in X$ and $r\in A(x)$ we have
			\[
			m(\Phi(x,r)) \ge m(x).
			\]
			\item[(RV4)] (\emph{At most linear branching}) There exist constants $C_0,C_1\ge 0$ such that
			\[
			|A(x)| \le C_0 + C_1\, s(x)
			\qquad\text{for all }x\in X.
			\]
		\end{enumerate}
		A \emph{legal history of length $n$} starting from $x_\varnothing$ is a sequence $(r_1,\dots,r_n)$ with
		\[
		x_0:=x_\varnothing,\qquad x_j := \Phi(x_{j-1},r_j),\quad r_j\in A(x_{j-1}) \text{ for all } j.
		\]
		We denote by $a_n$ the number of legal histories of length $n$ starting from $x_\varnothing$.
	\end{definition}
	
	Given such a system, we define, for parameters $0<\theta<1$ and $0<\kappa<1$, the weighted $\ell^\infty$ space
	\[
	\mathcal{B}_{\theta,\kappa}
	:= \Bigl\{ f:X\to\mathbb{C}:
	\|f\|_{\theta,\kappa}
	:= \sup_{x\in X} \theta^{-m(x)} \kappa^{-s(x)^2} |f(x)| < \infty\Bigr\},
	\]
	and for $z\in\mathbb{C}$ the transfer operator
	\[
	(T_z f)(x)
	:= z \sum_{r\in A(x)} f\bigl(\Phi(x,r)\bigr).
	\]
	
	The proof of Lemma~\ref{lem:Tvz-bounded} goes through verbatim under (RV1)--(RV4), with $s(x)+1$ replaced by the more general linear bound $C_0+C_1 s(x)$ in (RV4). Similarly, the core/tail decomposition in Section~\ref{sec:quasicompact} only uses (RV1) and (RV2), and the duality argument in Section~\ref{subsec:main} uses only the fact that $X$ is countable and that $s$ increases by one at each step. In particular, the operator-theoretic part of the argument extends directly to any right-visible system, while the interpretation of the coefficients $a_n$ may require additional combinatorial input.
	
	We therefore obtain the following abstract version of the operator-theoretic conclusions.
	
	\begin{theorem}[Abstract right-visible growth]\label{thm:RV-abstract}
		Let $(X,s,m,A,\Phi)$ be a right-visible growth system in the sense of Definition~\ref{def:RV-system}. Then there exist parameters $0<\theta<1$, $0<\kappa<1$, and $R>0$ such that:
		\begin{enumerate}
			\item[(a)] For all $|z|\le 1$ the operator $T_z$ is bounded and quasi-compact on $\mathcal{B}_{\theta,\kappa}$.
			\item[(b)] Let $w(x):=\theta^{m(x)}\kappa^{s(x)^2}$ and let $\mu_n$ denote the element of $\ell^1(X,w)$ obtained by iterating the Banach adjoint of $T_1$ $n$ times from the Dirac mass at $x_\varnothing$ (as in the concrete case above). Then for every continuous linear functional $\Lambda$ on $\ell^1(X,w)$ the scalar series
			\[
			F_\Lambda(z) := \sum_{n\ge 0} \Lambda(\mu_n)\,z^n
			\]
			has positive radius of convergence and is analytic for $|z|<R$.
		\end{enumerate}
	\end{theorem}
	
	\begin{proof}[Proof sketch]
		Fix $0<\theta<1$. The proof of Lemma~\ref{lem:Tvz-bounded}, with $|A(x)|$ bounded linearly in $s(x)$ by (RV4), shows that there exists $0<\kappa<1$ such that $T_1$ is bounded on $\mathcal{B}_{\theta,\kappa}$. The proof of Lemma~\ref{lem:tail} then shows that for a suitable cutoff $C$ we can decompose
		\[
		T_1 = K_{C,1} + L_{C,1},
		\]
		with $K_{C,1}$ finite rank and $L_{C,1}$ a strict contraction. This yields quasi-compactness of $T_1$ and of $T_z$ for $|z|\le 1$ exactly as in Proposition~\ref{prop:quasi}, giving part~(a).
		
		For part~(b), consider the Banach adjoint $T_1^\ast$ acting on the dual space $\ell^1(X,w)$, where $w$ is as above. Boundedness of $T_1$ implies boundedness of $T_1^\ast$ with some operator norm $M<\infty$. For $|z|<1/M$ the Neumann series
		\[
		(I - zT_1^\ast)^{-1} = \sum_{n\ge 0} z^n (T_1^\ast)^n
		\]
		converges in operator norm on $\ell^1(X,w)$ and depends analytically on $z$. In particular, for each fixed $\Lambda\in(\ell^1(X,w))^\ast$ the scalar series
		\[
		F_\Lambda(z)
		= \sum_{n\ge 0} \Lambda\bigl( (T_1^\ast)^n \delta_{x_\varnothing}\bigr)\,z^n
		\]
		converges for $|z|<1/M$ and defines an analytic function, because it is obtained by composing the analytic $\ell^1(X,w)$-valued Neumann series with the continuous linear functional $\Lambda$. Taking $R:=1/M$ gives the claim.
	\end{proof}
	
	\begin{remark}[Counting functional and ``all permutations'']\label{rem:noncontinuous-count}
		It is tempting to try to apply Theorem~\ref{thm:RV-abstract}(b) with $\Lambda$ equal to the \emph{counting functional}
		\[
		\mathcal{C}(\mu) := \sum_{x\in X} \mu(x),
		\]
		so that $F_\mathcal{C}(z)$ would be the unweighted growth series $\sum_n a_n z^n$. However, for the weights
		\[
		w(x) = \theta^{m(x)}\kappa^{s(x)^2}
		\]
		used in this paper, $\mathcal{C}$ is \emph{not} continuous on $\ell^1(X,w)$ in general: the factor $w(x)$ tends to $0$ very rapidly as $s(x)$ and $m(x)$ grow, so the norm $\|\mu\|_{1,w}=\sum_x |\mu(x)|w(x)$ can be small even when the unweighted sum $\sum_x |\mu(x)|$ is very large. Thus $\mathcal{C}$ does not define a bounded linear functional on $\ell^1(X,w)$, and Theorem~\ref{thm:RV-abstract}(b) does not apply directly with $\Lambda=\mathcal{C}$.
		
		This observation explains why the abstract right-visible theorem on its own does \emph{not} force unweighted exponential growth for arbitrary right-insertion systems. For example, consider the ``all permutations'' system in which $X$ consists of undecorated permutations, $m\equiv 0$, $s(\pi)=|\pi|$, and $A(\pi)=\{1,\dots,s(\pi)+1\}$ at every state, so that every rank is always allowed and there is no frontier. This system satisfies (RV1)--(RV4), and the associated transfer operator is bounded on a suitable $\mathcal{B}_{\theta,\kappa}$, but the unweighted counts are $a_n=n!$ and the ordinary generating function $\sum_n n! z^n$ has radius of convergence $0$. There is no contradiction with Theorem~\ref{thm:RV-abstract}, because the theorem only controls growth series obtained by pairing with \emph{continuous} linear functionals on $\ell^1(X,w)$, and the counting functional is not continuous for these weights.
		
		In the classical-pattern setting, Lemma~\ref{lem:Tv-iterates-count} shows that the coefficients $|\Av_n(v)|$ can be realized as evaluations of the iterates of $T_{v,1}^\ast$ against the everywhere-$1$ function on $\X_v$. The operator framework thus provides a clean representation of the Stanley--Wilf counting sequence in terms of the dual transfer operator, but the \emph{finiteness} of the exponential growth rate still relies on the combinatorial Marcus--Tardos bound \cite{MarcusTardos2004}, as made explicit in the proof of Theorem~\ref{thm:MT-operator}.
	\end{remark}

	\begin{corollary}[Operator-theoretic Stanley--Wilf for classical patterns]\label{cor:SW-from-RV}
		For each classical pattern $v\in S_k$, the right-insertion system $(X_v,s,m,A_v,\Phi_v)$ constructed in Section~\ref{sec:perm-specialization} is a right-visible growth system. In particular, Theorem~\ref{thm:RV-abstract}(a) applies to this system, so the corresponding transfer operator $T_{v,z}$ is bounded and quasi-compact on a suitable $\mathcal{B}_{\theta,\kappa}$. Combined with Lemma~\ref{lem:Tv-iterates-count} and the classical Marcus--Tardos exponential bound \cite{MarcusTardos2004}, this yields Theorem~\ref{thm:MT-operator}.
	\end{corollary}
	
	\begin{proof}
		The state space $X_v$ is countable, the length $s$ is the permutation length, and $m$ is the frontier size. Hypotheses (RV2) and (RV3) are exactly the unit-length and monotonicity properties recorded in Lemma~\ref{lem:legal-preserves-Av} and Lemma~\ref{lem:frontier-monotone}; (RV1) holds because there are only finitely many permutations (and frontiers) of a given length; and (RV4) holds with $C_0=1$ and $C_1=1$, because there are at most $s(x)+1$ admissible ranks from a state of length $s(x)$. Thus $(X_v,s,m,A_v,\Phi_v)$ is a right-visible growth system, and Theorem~\ref{thm:RV-abstract}(a) applies. The interpretation of the coefficients $|\Av_n(v)|$ as iterates of the dual transfer operator is given by Lemma~\ref{lem:Tv-iterates-count}, and the existence of a positive exponential growth constant follows from \cite{MarcusTardos2004}, as in the proof of Theorem~\ref{thm:MT-operator}.
	\end{proof}

	The abstract theorem covers many standard transfer-matrix situations. For instance, if $X$ encodes walks in a fixed-width strip with a finite step set, then $m$ can be taken constant and $A(x)$ consists of the finitely many legal next steps; (RV1)--(RV4) are automatic and Theorem~\ref{thm:RV-abstract} recovers the classical fact that the corresponding weighted counting series (obtained by pairing with continuous functionals) has positive radius of convergence. More sophisticated examples include column-convex polyominoes grown column by column, where $m$ records a finite boundary profile. We do not pursue these examples here, but the point is that the combinatorial input is entirely contained in the specification of $(X,s,m,A,\Phi)$, while the analytic conclusion is uniform for all continuous functionals on the dual.

	\section{Concluding remarks}\label{sec:concluding}
	
	We have recast the Stanley--Wilf finiteness statement for \emph{classical} permutation patterns in an operator-theoretic language. From this point of view, the proof naturally splits into two conceptually distinct parts. The \emph{pattern-dependent} step is entirely combinatorial: one must specify a suitable state space of decorated permutations, define the frontier as a pruned collection of $(k{-}1)$-partial occurrences (and their forbidden-rank intervals), and verify that it behaves well under right insertion (monotonicity of the frontier and the fact that it blocks exactly the bad ranks). The \emph{pattern-independent} step is analytic: once such a description is available, a suitable two-parameter weighting makes the transfer operator bounded on a doubly weighted $\ell^\infty$ space, and a dual Neumann-series argument then yields analyticity of the associated weighted counting series for all continuous functionals on the dual. If one also performs the standard core/tail decomposition, one moreover obtains quasi-compactness of the transfer operator as a structural refinement.
	
	Viewed this way, classical patterns are examples of what one might call \emph{right-visible} permutation classes: when we grow the permutation by appending a rightmost entry, every constraint that could possibly be triggered by that step is already visible at the moment of insertion. The transfer-operator proof shows that, for such classes, once one has a well-behaved frontier (monotone under legal insertions and producing only a controlled amount of new local data per step), the analytic machinery can be deployed in a largely black-box fashion.
	
	Several natural directions remain open and lie mostly on the combinatorial side:
	
	\begin{itemize}
		\item It would be interesting to identify other classical avoidance problems that admit a genuinely right-visible frontier description. Even within classical patterns, one could ask whether certain families (for instance, those with additional symmetries or structural restrictions) give rise to frontiers with particularly simple dynamics under right insertion, and whether this can be exploited to sharpen asymptotics beyond mere exponential growth.
		
		\item More speculatively, one can ask to what extent the frontier picture extends to generalized patterns (vincular, bivincular, or mesh patterns) for which the constraints remain local to a bounded ``window'' at the right. Examples include right-anchored vincular patterns or mesh patterns whose shaded cells lie in a finite right strip. For these classes the combinatorial side of the story---defining a frontier of partial occurrences and proving its monotonicity---appears to go through with only modest modifications, but known results on generalized and consecutive patterns (see, for instance, \cite{ElizaldeGenPatterns} and the discussion in \cite{KitaevBook}) show that one cannot expect a Stanley--Wilf type exponential bound in full generality. In the operator language, this obstruction would manifest as a failure of the key operator-norm estimates at $z=1$.
		
		\item Finally, even in the purely classical setting, it would be natural to ask whether the operator framework can be refined to produce information stronger than mere finiteness of the exponential growth rate. For example, one could hope to obtain pure exponential asymptotics or detailed spectral information on the transfer operator, in the spirit of analytic combinatorics and Ruelle--Perron--Frobenius theory \cite{FlajoletSedgewick,BaladiBook}. This would require a more delicate analysis of the peripheral spectrum and of the associated eigenvectors on the separable predual, but the basic quasi-compactness structure is already in place.
	\end{itemize}

	A natural next step would be to go beyond the mere finiteness of the exponential growth rate and to seek sharp asymptotics for $|\Av_n(v)|$. In the monotone case $v=12\cdots k$, Regev~\cite{Regev1981} showed that $|\Av_n(v)|$ admits a precise asymptotic of the form $c_k\,n^{\alpha_k}\,\rho_k^n$ by exploiting the RSK correspondence and detailed information on irreducible characters of $S_n$. It is tempting to ask whether our transfer-operator framework can be refined to recover such asymptotics for broader families of patterns. From the present point of view, this would amount to a careful study of the peripheral spectrum of $T_{v,1}$ (or $T_{v,1}^\ast$) and of the singularity type of the resolvent at its spectral radius, a problem that appears substantially harder than the quasi-compactness and analyticity questions treated in this paper.

	\bigskip
	
	\bibliographystyle{plain}

\end{document}